\documentclass{article}
\usepackage{amssymb,amsfonts,amsmath,amsthm,enumerate}
\usepackage[alphabetic]{amsrefs}

\def \a{\mathfrak a}
\def \A{{\mathbb A}}
\def \al{\alpha}
\def \C{{\mathbb C}}
\def \CO{\mathcal O}
\def \e{\emph}
\def \eps{\varepsilon}
\def \F{{\mathbb F}}
\def \G{{\mathbb G}}
\def \Ga{\Gamma}
\def \GL{\operatorname{GL}}
\def \Hom{\operatorname{Hom}}
\def \la{\lambda}
\def \N{{\mathbb N}}
\def \om{\omega}
\def \p{\mathfrak p}
\def \R{{\mathbb R}}
\def \Re{\operatorname{Re}}
\def \SL{\operatorname{SL}}
\def \sm{\smallsetminus}
\def \Sp{\operatorname{Sp}}
\def \spec{\operatorname{spec}}

\def \Z{{\mathbb Z}}
\def \({\left(}
\def \){\right)}

\newtheorem{theorem}{Theorem}[section]

\newtheorem{lemma}[theorem]{Lemma}
\newtheorem{corollary}[theorem]{Corollary}
\newtheorem{proposition}[theorem]{Proposition}
\numberwithin{equation}{section} \theoremstyle{definition}
\newtheorem{definition}[theorem]{Definition}
\newtheorem{example}[theorem]{Example}

\begin{document}

\title{Counting and  zeta functions over $\F_1$\\ \ \\ \small
Abh. Math. Sem. Hamburg. Vol. 85, Issue 1, 59-71 (2015)}
\author{Anton Deitmar\footnote{Universit\"at T\"ubingen, Mathematisches Institut,
Auf der Morgenstelle 10, 72076 Tuebingen, Germany. {The first named author stayed at the SFB 878 at M\"unster while this paper was finished. He wants to thank the people there for there hospitality.}}, Shin-ya Koyama\footnote{Department of Biomedical Engineering, Toyo University,
2100 Kujirai, Kawagoe, Saitama, 350-8585, Japan.} \ \& Nobushige Kurokawa\footnote{Department of Mathematics, Tokyo Institute of Technology, 
Oh-okayama, Meguro-ku, Tokyo, 152-8551, Japan.}}

\maketitle

\section*{Introduction}
This paper is dedicated to counting functions and the corresponding zeta functions in $\F_1$-theory.
The first $\F_1$-zeta function has been defined by Christophe Soul\'e in \cite{Soule} for schemes of finite type whose congruence counting function is a polynomial (Soul\'e's condition), which makes it possible to replace the prime $p$ by $p=1$ yielding the $\F_1$-zeta function as limit.
In the paper \cite{DzetaK} it is shown that if a finite type scheme $X$ is defined over $\F_1$, then it satisfies  Soul\'e's condition up to torsion in the structure sheaf.
In the paper \cite{CCabsolutePoint}, Alain Connes and Caterina Consani propose a generalization of Soul\'e's zeta function which incorporates torsion.
The result is a transcendental function which is hard to compute explicitly.
The construction of Connes and Consani applies to counting functions of the form $N(q)=\sum_{j=1}^n c_j(q)q^j$, where each $c_j(q)$ is a periodic function.
The arguments $q$ are prime powers.
The central technical step in \cite{CCabsolutePoint} is a construction of an extension of $c_j(q)$ from integers $q$ to a periodic function on the reals.
There then occur natural constraints which lead to the transcendental construction.
In the first section of this paper we propose a different path.
As it turns out, the functions $n\mapsto c_j(p^n)$ are also periodic for any prime $p$. As the period is independent of $p$, one can extend these functions in a unified way, paving the path for taking the limit $p\to 1$.
In this way one gets a calculus of $\F_1$-zeta functions incorporating torsion, yielding  rational functions, which satisfy natural functional equations.

In the second section we consider Soul\'e zeta functions for reductive groups and compute their functional equations, see \cite{Lor}.
In the third section we investigate the correlation between functional equations of the zeta function and its counting function, the latter considered as a function of real arguments.
As the zeta function doesn't determine a real argument counting function, one needs restrictions, which in this case are given by the consideration of finite sums of real powers.
The decisive technical tool we put forward here is a new type of regularization of the zeta integral through a two variable zeta integral.
This regularization is an adaptation of the known method of zeta regularization for determinants of elliptic differential operators as in \cites{RS,DHoker,Sarnak}.

\section{Refined Soul\'e zeta functions}\label{sec1}
Let $X$ be a scheme of finite type over $\Z$.
For a prime $p$ one sets after Weil,
$$ 
Z_X(p,T)=\exp\(\sum_{n=1}^\infty\frac{T^n}n\# X(\F_{p^n})\),
$$
where $\F_{p^n}$ stands for the field of $p^n$ elements.
This is the local zeta function of $X$.
The global is
$$
\zeta_{X/\Z}(s)=\prod_p Z_X(p,p^{-s}),
$$
where the product runs over all primes $p$.
The scheme $X$ is said to satisfy \e{Soul\'e's condition},  (see \cite{Soule}), if there exists a polynomial $N(x)$ with integer coefficients such that $\# X(\F_{p^n})=N(p^n)$ holds for every prime $p$ and every $n\in\N$.
Then $Z_X(p,p^{-s})$ is a rational function in $p$ and $p^{-s}$. The pole order at $p=1$ is $N(1)$.
One may then define the \e{Soul\'e zeta function} as
$$
\zeta_{X/\F_1}(s)=\lim_{p\to 1}\Z_X(p,p^{-s})(p-1)^{N(1)}.
$$
One computes that if $N(x)=a_0+\dots+a_nx^n$, then
$$
\zeta_{X/\F_1}(s)=s^{-a_0}(s-1)^{-a_1}\cdots (s-n)^{-a_n}.
$$
In the paper \cite{DzetaK} it is shown that if a finite type scheme $X$ is defined over $\F_1$, then it satisfies the Soul\'e condition up to torsion in the structure sheaf.
In the paper \cite{CCabsolutePoint}, Alain Connes and Caterina Consani propose a generalization of Soul\'e's zeta function which incorporates torsion.
The result is a transcendental function which cannot be computed directly.
The central technical step in \cite{CCabsolutePoint} is a construction for an extension to $\R$ of a periodic function on $\Z$.
This is applied to a counting function which occurs after taking the limit $p\to 1$.

In this section we reverse the order of steps in that we first extend and then take the limit $p\to 1$.
This leaves us free from technical constraints and allows us to choose the simplest and most natural way of extension: via the Fourier series.
With this choice one gets a calculus of zeta functions which
\begin{itemize}
\item extends the torsion-free case,
\item is expressible by Betti numbers exactly as in the torsion-free case, and
\item satisfies the same functional equation as in the torsion-free case. 
\end{itemize}

In order to set the stage and introduce notation, we briefly repeat the basic notions of $\F_1$-schemes of \cite{DF1}, except that, for better distinction from other concepts, we now call them \e{monoid schemes}.

An \e{ideal} in a commutative monoid $A$ is a subset $\a\subset A$ with $A\a\subset\a$.
A \e{prime ideal} is an ideal $\p$ such that $S_\p= A\sm\p$ is a submonoid.
The \e{spectrum} is the set $\spec(A)$ of all prime ideals with the topology generated by all sets of the form $D(f)=\{ \p:f\notin \p\}$.
It carries a canonical sheaf $\CO_{A}$ of monoids with stalks $\CO_{A,\p}=A_\p=S_\p^{-1} A$.
The pair $(\spec A,\CO_A)$ is then called an \e{affine monoid scheme}.
A  \e{monoid scheme} is a topological space $X$ with a sheaf $\CO_X$ of monoids which is locally affine.

An affine scheme is given by a monoid.
Its \e{$\Z$-lift} is given by the corresponding monoidal ring.
This procedure is compatible with gluing and thus extends to schemes to give the the \e{base change} $X\mapsto X_\Z$ that assigns a scheme over $\Z$ to any monoidal scheme $X$.
A $\Z$-scheme isomorphic to the lift of a monoidal scheme is said to be \e{defined over $\F_1$}.
In \cite{toric} it is shown that a variety is defined over $\F_1$ if and only if it is a toric variety.

A monoidal scheme $X$ is said to be of \e{finite type}, if there exists a finite affine covering $X=\bigcup_{j=1}^nU_j$ such that each monoid $\CO_X(U_j)$ is finitely generated.
In this case, $X$ is a finite set. We assume $X$ to be of finite type from now on.
For a given natural number $m$ let $F_m$ denote the monoid $\mu_m\cup\{ 0\}$, where $\mu_m$ is the cyclic group of order $m$ and $a\cdot 0=0$ for every $a$.
By Theorem 1.1 in \cite{DF1} one has $X_\Z(\F_{q})\cong X(F_{q-1})$ for every prime power $q$ and by Lemma 1 of \cite{DKK},
$$
X(F_m)\cong\coprod_{x\in X}\Hom(\CO_{X,x}^\times,\mu_m)
$$
for every $m\in\N$.
Now $\CO_{X,x}^\times$ is a finitely generated abelian group, hence it is a product of cyclic groups $\CO_{X,x}^\times\cong C^{R(x)}\times \mu_{t_{x,1}}\times\dots\times\mu_{t_{x,k}}$, where $C$ is the infinite cyclic group, $R(x)$ is the rank of the group $\CO_{X,x}^\times$, and the numbers $t_{x,j}$ can be assumed to be prime powers.
For simplicity of notation, we use a single $k$ here, so some of the groups $\mu_{t}$ may be trivial.
Writing $m=q-1$ it turns out
$$
\# X_\Z(\F_q)=\sum_{x\in X}\prod_{j=1}^k\underbrace{\#\Hom(\mu_{t_{x,j}},\mu_m)}_{=\gcd(t_{x,j},m)} m^{R(x)}
$$
We conclude
\begin{align*}
Z_{X_\Z}(p,p^{-s})&=\exp\(\sum_{n=1}^\infty \frac{p^{-ns}}n \#X_\Z(\F_{p^n})\)\\
&=\exp\(\sum_{n=1}^\infty \frac{p^{-ns}}n 
\sum_{x\in X}(p^n-1)^{R(x)}\prod_{j=1}^k\gcd(t_{x,j},p^n-1)
\).
\end{align*}
We now argue that the function $n\mapsto \gcd(t_{x,j},p^n-1)$ is periodic.
For this note that $\gcd(t,p^n-1)$ only depends on the residue class of $p^n-1$ modulo $t$.
Assume that $t=q^k$ is a prime power.
If $q\ne p$, then $p^n$ is a unit modulo $t=q^k$ and the map $n\mapsto p^n-1\mod(t)$ is periodic of period 
$\phi(t)=\#(\Z/t\Z)^\times$.
If $p=q$, then the function $n\mapsto \gcd(t,p^n-1)$ is constantly equal to $1$.

We have that in any case the map $n\mapsto\gcd(t_{x,j},p^n-1)$ is periodic of period $\phi(t_{x,j})$.
Let $n_0$ be the least common multiple of all the periods $\phi(t_{x,j})$ as $x$ and $j$ vary, and set $\xi=e^{2\pi i/n_0}$.
Note that $n_0$ does not depend on $p$.
For $\nu=1,\dots n_0$ let
$$
c_{x,j,\nu}(p)=\frac1{n_0}\sum_{n=1}^{n_0}\gcd(t_{x,j},p^n-1)\xi^{-n\nu}
$$
be the Fourier coefficient, then
$$
\gcd(t_{x,j},p^n-1)=\sum_{\nu=1}^{n_0} c_{x,j,\nu}(p)\xi^{n\nu}
$$
is the Fourier series in the variable $n$.
So $Z_{X_\Z}(p,p^{-s})$ equals
$$
\exp\(\sum_{n=1}^\infty\frac{p^{-ns}}n
\sum_{x\in X}(p^n-1)^{R(x)}\prod_{j=1}^k\(\sum_{\nu=1}^{n_0}c_{x,j,\nu}(p)\xi^{n\nu}\)\).
$$
We can now pull out the sum over $x\in X$, so that it becomes a product, replace $(p^n-1)^{R(x)}$ with the sum $\sum_{r=0}^{R(x)}\left(\begin{array}{c}R(x) \\ r\end{array}\right)p^{nr}(-1)^{R(x)-r}$ and pull out the sum again and turn the integer factors into exponents.
After that, we can multiply out the product over $j$ and repeat the procedure.
We end up with a product of exponential factors that come with exponents which involve products of the coefficients $c_{x,j,\nu}(p)$.
We ultimately want to let $p$ tend to 1 and we know how to do that with all terms except for $c_{x,j,\nu}(p)$. 
In the definition of this coefficient  there occurs the factor $\gcd(t,p^n-1)$. The map $m\mapsto \gcd(t,m)$  is periodic  of period $t$, hence it has a Fourier expansion
$$
\gcd(t,p^n-1)=\sum_{\alpha=1}^t d_\alpha e^{2\pi i\alpha(p^n-1)/t},
$$
which tends, as $p\to 1$, to 
\begin{align*}
\sum_{\alpha=1}^t d_\alpha
&= \sum_{\al=1}^t\frac1t\sum_{j=1}^t\gcd(t,j)e^{-2\pi ij\al/t}\\
&= \frac1t\sum_{j=1}^t\gcd(t,j)\sum_{\al=1}^t e^{-2\pi ij\al/t}=\frac1t\sum_{j=1}^t\gcd(t,j)
\times\left\{\begin{array}{cc} t& j=t,\\ 0&j\ne t.\end{array}\right\}\\
&= t.
\end{align*}
Hence $c_{x,j,\nu}(p)$ tends, as $p\to 1$, to
$$
\frac1{n_0}\sum_{n=1}^{n_0}t_{x,j}\xi^{-n\nu}=\begin{cases}t_{x,j}&\nu=n_0,\\ 0&\nu\ne n_0.\end{cases}
$$
This means that, as we intend to let $p$ tend to 1, we can replace the coefficient $c_{x,j,\nu}(p)$ with this result and consider
$$
\tilde Z_{X_\Z}(p,p^{-s})
=\exp\(\sum_{n=1}^\infty\frac{p^{-ns}}n
\sum_{x\in X}(p^n-1)^{R(x)}\prod_{j=1}^kt_{x,j}\).
$$
The number $T(x)=\prod_{j=1}^k t_{x,j}$ equals the cardinality of the torsion group of $\CO_{X,x}^\times$.
We arrive at
$$
\tilde Z_{X_\Z}(p,p^{-s})
=\prod_{x\in X}\prod_{r=0}^{R(x)}(1-p^{r-s})^{T(x)\binom{R(x)}r (-1)^{R(x)-r-1}}.
$$
The pole order at $p=1$ is $N=\sum_{x\in X}\sum_{r=0}^{R(x)}T(x)\binom{R(x)}r (-1)^{r-R(x)}$ and
\begin{align*}
\zeta_{X/\F_1}(s)
&=\lim_{p\to 1}Z_{X_\Z}(p,p^{-s})(p-1)^N\\
&=\lim_{p\to 1}\tilde Z_{X_\Z}(p,p^{-s})(p-1)^N\\
&=\prod_{x\in X}\prod_{r=0}^{R(x)}(s-r)^{T(x)\binom{R(x)}r (-1)^{R(x)-r-1}}.
\end{align*}
Let $R$ be the maximal value of $R(x)$, then we can rewrite this as
$$
\zeta_{X/\F_1}(s)
=\prod_{r=0}^{R}(s-r)^{\sum_{x\in X}T(x)\binom{R(x)}r (-1)^{R(x)-r-1}}.
$$
We have shown the first part of the following theorem.

\begin{theorem}\label{thm1.1}
Let $X$ be a monoid scheme of finite type.
For each $x\in X$ let $R(x)$ denote the rank of the finitely generated abelian group $\CO_{X,x}^\times$, let $T(x)$ be the order of its torsion group and write $R$ for the maximum of all $R(x)$ for $x\in X$.
With the normalization that a periodic function on $\Z$ is extended to $\R$ via its Fourier series, one gets
$$
\zeta_{X/\F_1}(s)
=\prod_{r=0}^{R}(s-r)^{\sum_{x\in X}T(x)\binom{R(x)}r (-1)^{R(x)-r-1}}.
$$
If $X_\Z$ is a smooth projective variety of dimension $d$, the zeta function satisfies the functional equation
$$
\zeta_{X/\F_1}(d-s)=(-1)^\chi\zeta_{X/\F_1}(s),
$$
where $\chi$ is the Euler characteristic of $X_\Z$.
\end{theorem}

The functional equation is the same as in the torsion-free case \cite{Lor}.

\begin{proof}
It remains to show the functional equation.
By Deligne's proof of the Weil conjectures \cite{Deligne} we have for every prime $p$ that $Z_{X_\Z}(p,T)$ equals the product $\prod_{\nu=0}^{2d}P_\nu(T)^{(-1)^{\nu+1}}$ with $P_\nu(T)=\prod_{j=1}^{b_\nu}(1-\al_{\nu,j}T)$, where $|\al_{\nu,j}|=p^{\nu/2}$ and $b_{2d-\nu}=b_\nu$.
For any choice $(\nu)$ of $1\le \nu_1,\dots,\nu_{k}\le n_0$ we set $c_x^{(\nu)}(p)=\prod_{j=1}^kc_{x,j,\nu_j}(p)$.
With this notation,
$$
\prod_{j=1}^k\(\sum_{\nu=1}^k c_{x,j,\nu}(p)\xi^{n\nu}\)
=\sum_{(\nu)}c_x^{(\nu)}(p)\xi^{n\sum(\nu)},
$$
where $\sum(\nu)=\sum_{j=1}^k\nu_j$.
So we get that $Z_{X_\Z}(p,T)$ equals
$$
\prod_{x\in X}\prod_{r=0}^{R(x)}\prod_{(\nu)}\(1-p^r\xi^{\sum(\nu)}T\)^{(-1)^{R(x)-r-1}\binom{R(x)}rc_x^{(\nu)}(p)}.
$$
The factor $p^r\xi^{\sum(\nu)}$ has absolute value $p^r$.
Comparing this with Deligne's results we see that $b_{2r+1}=0$ and
$$
b_{2r}=\sum_{x\in X}\sum_{(\nu)}(-1)^{R(x)-r}\binom{R(x)}rc_x^{(\nu)}(p),
$$
so the right hand side does not depend on $p$.
Recall that
$$
c_{x,j,\nu}(p)=\frac1{n_0}\sum_{n=1}^{n_0}\gcd(t_{x,j},p^n-1)\xi^{-n\nu}.
$$
By Dirichlet's prime number theorem there exists a prime $p$ such that $p\equiv 1\mod(t_{x,j})$ for all $x,j$.
For such $p$ one gets
$$
c_{x,j,\nu}(p)=\begin{cases}t_{x,j}&\nu=n_0,\\
0&\nu\ne n_0.\end{cases}
$$
and so
$$
c_x^{(\nu)}(p)=\begin{cases}T(x)&\nu_1=\dots=\nu_k=n_0,\\
0&\text{otherwise}.\end{cases}
$$
Which means that
$$
b_{2l}=\sum_{x\in X}(-1)^{l+R(x)}\binom{R(x)}lT(x),
$$
so that we have
$$
\zeta_{X/\F_1}(s)=\prod_{l=0}^d(s-l)^{-b_{2l}}.
$$
As $b_{2l}=b_{2d-2l}$, the claimed functional equation follows from this.
\end{proof}

\begin{corollary}
If $X_\Z$ is a smooth projective variety, then for every prime $p$ one has the functional equation
$$
Z(p,\frac1{p^dT})= (-1)^\chi p^{d\chi/2}T^\chi Z(p,T).
$$
\end{corollary}

This corollary can of course also be proved directly, i.e., without use of $\F_1$-theory.

\begin{proof}
Deligne has proved that
$
Z(p,\frac1{p^dT})= \pm p^{d\chi/2}T^\chi Z(p,T),
$
where the sign is $+$ if $d$ is odd and $(-1)^{m+b_d}$ if $d$ is even, where $m$ is the multiplicity of the eigenvalue $-p^{d/2}$.
From the proof of Theorem \ref{thm1.1} we deduce that $m$ does not depend on $p$ and therefore the sign in the functional equation is the same for all $p$, i.e., replacing $T$ with $p^{-s}$ we see that there exists $\eps\in\{1,-1\}$ such that
$
Z(p,p^{s-d})= \eps p^{d\chi/2}p^{-s\chi} Z(p,p^{-s}),
$
and from this $\zeta_{X/\F_1}(d-s)=\eps\zeta_{X/\F_1}(s)$, so by the theorem it follows $\eps=(-1)^\chi$, whence the corollary.
\end{proof}

\section{Reductive groups over $\F_1$}
An example of varieties which satisfy Soul\'e's condition, but are not defined over $\F_1$, are split reductive groups like $\GL_n,\SL_n,\Sp_{2n}$ or the split quadratic groups.
In this section, we fix a misprint in the paper \cite{Lor}. For the convenience of the reader, we repeat the notation and principal assertions of the latter.

\begin{proposition}\label{prop2.1}
Let $G$ denote a split reductive group over $\Z$.
Then $G$ satisfies Soul\'e's condition.
The conting polynomial $N_G(q)$ satisfies the functional equation
$$
N_G\(\frac1q\)=(-1)^rq^{-d-N}N_G(q),
$$
or, equivalently, the $\F_1$-zeta function $\zeta_G(s)$ satisfies the functional equation
$$
\zeta_{G/\F_1}(d+p-s)=(-1)^\chi\zeta_{G/\F_1}(s)^{(-1)^r},
$$
where $d$ is the dimension, $r$ the rank of $G$ and $p=(d-r)/2$ is the number of positive roots of $G$.
\end{proposition}

\begin{proof}
The fact that Soul\'e's condition is satisfied is well known.
Let $B\subset G$ be a Borel subgroup and $T\subset B$ a maximal torus.
Then $B$ is, as a scheme, isomorphic to $\GL_1^r\times\A^p$.
By Section \ref{sec1} one sees that the quotient variety $G/B$ is smooth projective with counting function $N_{G/B}(q)=\sum_{l=0}^pb_{2l}q^l$, where $b_{2l}$ is the Betti number satisfying $b_{2l}=b_{2(N-l)}$.
Putting things together, one has
\begin{align*}
N_G(q)&=N_B(q)N_{G/B}(q)\\
&=(q-1)^rq^p\(\sum_{l=0}^pb_{2l}q^l\)\\
&=\sum_{k=p}^dq^k
\underbrace{\(\sum_{j+p+l=k}\binom rj(-1)^{r-j}b_{2l}\)}_{=a_k}
\end{align*}
Using $b_{2l}=b_{2(p-l)}$ on the one hand and $\binom rj=\binom r{r-j}$ on the other, we get
\begin{align*}
a_k&=\sum_{j+p+l=k}\binom rj(-1)^{r-j}b_{2l}\\
&=\sum_{j+p+l=k}\binom r{r-j}(-1)^{r-j}b_{2(p-l)}\\
&=(-1)^r\sum_{r-j+p+(p-l)=k}\binom r{j}(-1)^{r-j}b_{2l}\\
&= (-1)^r a_{d+p-k},
\end{align*}
which is equivalent to the claimed functional equation.
In the last line we have used that the condition $r-j+p+(p-l)=k$ is equivalent to $j+p+l=r+3p-k=d+p-k$.
\end{proof}

\section{Dual counting functions}
Consider a variety $X$ satisfying Soul\'e's condition.
So the counting function $N(x)$ is a polynomial, say $N(x)=a_0+a_1x+\dots+ a_nx^n$.
Then Soul\'e's zeta function is $\zeta_N(s)=\prod_{j=0}^n(s-j)^{-a_j}$.
In this case, a functional equation of $\zeta_N$ is equivalent to a corresponding functional equation of $N$, as becomes clear for instance from Proposition \ref{prop2.1}.
Since we are interested in extending the zeta calculus to cases when the counting function $N$ is no longer a polynomial, we need to reconsider the mechanism giving $\zeta_N$ out of $N$.
We consider the logarithmic derivative,
$$
\frac{\zeta_N'(s)}{\zeta_N(s)}=-\sum_{j=0}^n \frac{a_j}{s-j}=-\int_1^\infty N(u)u^{-s}\frac{du}u.
$$
Formally integrating gives
$$
\log(\zeta_N(s))=\int_1^\infty \frac{N(u)}{\log u}u^{-s}\frac{du}u, 
$$
or
$$
\zeta_N(s)=\exp\(\int_1^\infty \frac{N(u)}{\log u}u^{-s}\frac{du}u\).
$$
There is, however, a problem with this integral, as it only converges if $N(1)=0$.

If $N(1)\ne 0$, the integral needs to be regularized.
A crucial observation is that in the case under consideration, the double variable zeta function
$$
Z_N(w,s)=\frac1{\Gamma(w)}\int_1^\infty\frac{N(u)}{u^{s+1}}(\log u)^{w-1}du
$$
is regular at $w=0$ and that the identity
$
\zeta_N(s)=\exp\(\left.\frac\partial{\partial w}Z_N(w,s)\right|_{w=0}\)
$
holds.

\begin{definition}
A measurable function
$$
N:\ (1,\infty)\longrightarrow \C
$$
is called \e{admissible}, if the zeta integral
$$
Z_N(w,s)=\frac1{\Gamma(w)}\int_1^\infty\frac{N(u)}{u^{s+1}}(\log u)^{w-1}du
$$
converges for $\Re(s)> C$ for some $C>0$ and $w$ in some open domain such that $Z_N$ possesses a unique holomorphic extension to $w=0$.
In that case we define the zeta function as
$$
\zeta_N(s)=\exp\(\left.\frac\partial{\partial w}Z_N(w,s)\right|_{w=0}\).
$$
\end{definition}

Now suppose that $N$ is even defined on the interval $(0,\infty)$. Then we define the \e{dual counting function} by
$$
N^\ast(u)=N\(\frac1u\).
$$

\begin{definition}
Suppose that $\zeta_N(s)$ and $\zeta_{N^\ast}(s)$ have memormorphic continuation
to $\C$. Define the $\eps$ factor by
$$
\eps_N(s)=\frac{\zeta_{N^\ast}(-s)}{\zeta_{N}(s)}.
$$
\end{definition}

We shall next investigate, how these $\eps$-factors relate to functional equations of $\zeta_N$.

\begin{example}\label{ex3.3}
Let $N(u)=u^\alpha$ for $\alpha\in\R$. Then
$
Z_N(w,s)=(s-\alpha)^{-w}
$
and
$
\zeta_N(s)=\frac1{s-\alpha}.
$
Similarly
$
\zeta_{N^\ast}(s)=\frac1{s+\alpha}
$
and so
$$
\eps_N(s)=-1.
$$
\end{example}

\begin{lemma}\label{lem3.4}
If $N_1,N_2$ are admissible counting functions, then so is $N_1+N_2$ and one has
$$
\zeta_{N_1+N_2}=\zeta_{N_1}\zeta_{N_2}\quad\text{and}\quad \eps_{N_1+N_2}=\eps_{N_1}\eps_{N_2}.
$$
\end{lemma}

\begin{proof}
Straightforward by the definitions.
\end{proof}

\begin{proposition}
If $N$ is continuously differentiable and $N(1)=0$, then we have
\begin{align*}
\zeta_{N^\ast}(-s)&=\exp\(-\int_0^1\frac{N(u)}{u^{s+1}\log u}du\)
\tag*{$\Re(s)<1$}
\end{align*}
and
\begin{align*}
\zeta_N(s)^{-1}=\exp\(-\int_1^\infty\frac{N(u)}{u^{s+1}\log u}du\).
\tag*{$\Re(s)>\deg N$}
\end{align*}
\end{proposition}

\begin{proof}
We have that 
\begin{align*}
\frac\partial{\partial w}Z_N(w,s)&=
\(\frac1{\Ga(w)}\)'\int_1^\infty\frac{N(u)}{u^{s+1}}(\log u)^{w-1}\,du\\
&\ \ \ \ +\frac1{\Ga(w)}\int_1^\infty\frac{N(u)}{u^{s+1}}\log(\log u)(\log u)^{w-1}\,du.
\end{align*}
The second summand tends to zero as $w\to 0$ and the first tends to 
$\int_1^\infty\frac{N(u)}{u^{s+1}\log u}\,du$ giving the second claim.
Replacing $N$ by $N^*$ and $s$ by $-s$ yields
\begin{align*}
\zeta_{N^*}(s) &= \exp\(\int_1^\infty\frac{N(1/u)}{u^{-s}\log u}\,\frac{du}u\)\\
&=\exp\(-\int_0^1\frac{N(u)}{u^{s}\log u}\,\frac{du}u\)
\end{align*}
and thus the first claim.
\end{proof}

\begin{proposition}\label{prop3.6}
Suppose that
$
N(u)=\sum_{\alpha\in\C} m_\alpha u^\alpha
$
is a finite sum of powers with integral coefficients $m_\alpha\in\Z.$ Then
$$
\zeta_N(s)=\prod_{\al\in\C}\(\frac1{s-\al}\)^{m_\al}\quad\text{and}\quad
\zeta_{N^*}(s)=\prod_{\al\in\C}\(\frac1{s+\al}\)^{m_\al}
$$
so that $\eps_N(s)=(-1)^{N(1)}$.
If $N\ne 0$ satisfies a functional equation $N(1/u)=cu^{-\om} N(u)$ for some $c,\om\in\C$, then $c=\pm 1$ and the zeta function satisfies the functional equation
$$
\zeta_N(\om-s)=(-1)^{N(1)}\zeta_N(s)^c.
$$
\end{proposition}

\begin{proof}
The first assertion follows from Example \ref{ex3.3} and Lemma \ref{lem3.4}.
Suppose now that $N(1/u)=cu^{-\om}N(u)$.
Replacing $u$ with $1/u$ and iterating the functional equation gives $N(u)=c^2N(u)$ so that $N\ne 0$ yields $c=\pm 1$.
Next note that $N(1/u)=cu^{-\om}N(u)$ is equivalent to $m_{\om-\al}=cm_\al$, so that
\begin{align*}
\zeta_N(\om-s)&=\prod_{\al\in\C}\(\frac1{\om-s-\al}\)^{m_\al}\\
&=(-1)^{N(1)}\prod_{\al\in\C}\(\frac1{s-(\om-\al)}\)^{m_\al}\\
&=(-1)^{N(1)}\prod_{\al\in\C}\(\frac1{s-\al}\)^{m_{\om-\al}}\\
&=(-1)^{N(1)}\prod_{\al\in\C}\(\frac1{s-\al}\)^{cm_{\al}}=(-1)^{N(1)}\zeta_N(s)^c.\tag*\qedhere
\end{align*}

\end{proof}

\begin{example}
Let
$$
N(u)=(1-u^{-\omega_1})\cdots(1-u^{-\omega_r}).
$$
Then $\eps_N(s)=1$ and
$$
\zeta_N(-(\omega_1+\cdots+\omega_r)-s)
=\zeta_N(s)^{(-1)^r}.
$$
\end{example}

\begin{proposition}\label{Kurokawa3}
Let
$
N(u)=(1-u^{-1})^r
$
for an integer $r\ge 1$. Then
\begin{enumerate}[{\rm(a)}]
\item $\zeta_N(s)=\zeta_{\G_m^r/\F_1}(s+r)$.
\item $\zeta_{N^{\ast}}(s)=\zeta_{\G_m^r/\F_1}(s)^{(-1)^r}$.
\item $\zeta_{\G_m^r/\F_1}(r-s)=\zeta_{\G_m^r/\F_1}(s)^{(-1)^r}$.
\end{enumerate}
\end{proposition}

\begin{proof}
(a) From
$N_{\G_m^r}(u)=(u-1)^r=u^r N(u)$
with 
$
N(u)=(1-u^{-1})^r,
$
we obtain
\begin{align*}
\zeta_{\G_m^r/\F_1}(s)
&=\zeta_{N_{\G_m^r}}(s)=\zeta_N(s-r).
\end{align*}
(b)
Direct calculations give:
$N^\ast(u)=(1-u)^r=(-1)^r(u-1)^r=(-1)^r N_{\G_m^r}(u)$.
Hence
\begin{align*}
\zeta_{N^\ast}(s)&=\zeta_{N_{\G_m^r}}(s)^{(-1)^r}=\zeta_{\G_m^r/\F_1}(s)^{(-1)^r}.
\end{align*}
(c)
Proposition \ref{prop3.6} gives
$
\zeta_{N^\ast}(-s)=\zeta_N(s),
$ as in this case $\eps_N(s)=1$.
Hence, from (a) and (b) we get
$$
\zeta_{\G_m^r/\F_1}(-s)^{(-1)^r}=\zeta_{\G_m^r/\F_1}(s+r),
$$
which leads to
$\zeta_{\G_m^r/\F_1}(r-s)=\zeta_{\G_m^r/\F_1}(s)^{(-1)^r}.
$
\end{proof}

\begin{proposition}\label{Kurokawa4}
Let
$$
N(u)=(1-u^{-1})(1-u^{-2})\cdots(1-u^{-r}).
$$
Then
\begin{enumerate}[{\rm(a)}]
\item $\zeta_N(s)=\zeta_{\GL(r)/\F_1}(s+r^2)$.
\item $\zeta_{N^{\ast}}(s)=\zeta_{\GL(r)/\F_1}\(s+\frac{r(r-1)}2\)^{(-1)^r}$.
\item $\zeta_{\GL(r)/\F_1}\(\frac{r(3r-1)}2-s\)=\zeta_{\GL(r)/\F_1}(s)^{(-1)^r}$.
\end{enumerate}
\end{proposition}

\begin{proof}
(a)
From
$N_{\GL(r)}(u)=u^{r^2} N(u)
$ with
$
N(u)=(1-u^{-1})\cdots (1-u^{-r}),
$
we get
\begin{align*}
\zeta_{\GL(r)/\F_1}(s)
&=\zeta_{N_{\GL(r)}}(s)\\
&=\zeta_N(s-r^2).
\end{align*}
(c)
The identities
\begin{align*}
N^\ast(u)
&=(1-u)(1-u^2)\cdots(1-u^r)\\
&=(-1)^r u^{\frac{r(r+1)}2}(1-u^{-1})\cdots(1-u^{-r})\\
&=(-1)^r u^{-\frac{r(r-1)}2}N_{\GL(r)}(u)
\end{align*}
give
\begin{align*}
\zeta_{N^\ast}(s)
&=\zeta_{N_{\GL(r)}}\(s+\frac{r(r-1)}2\)^{(-1)^r}\\
&=\zeta_{\GL(r)/\F_1}\(s+\frac{r(r-1)}2\)^{(-1)^r}.
\end{align*}
(b)
From Theorem \ref{prop3.6} it holds that $\eps_N(s)=1$ and that
$\zeta_{N^\ast}(-s)=\zeta_N(s)$.
Hence
$$
\zeta_{\GL(r)/\F_1}\(\frac{r(r-1)}2-s\)^{(-1)^r}
=\zeta_{\GL(r)/\F_1}\(s+r^2\).
$$
Thus 
\begin{align*}
\zeta_{\GL(r)/\F_1}\(\frac{r(3r-1)}2-s\)
=\zeta_{\GL(r)/\F_1}\(s\)^{(-1)^r}.\tag*\qedhere
\end{align*}
\end{proof}

\begin{proposition}\label{Kurokawa5}
Let
$$
N(u)=\sum_{\lambda,m}c(\lambda,m)u^\lambda(\log u)^m
$$
be a finite sum with $\lambda\in\C$, $m\in\Z_{\ge0}$, and
$c(\lambda,m)\in\Z$.
Then the following holds:
\begin{enumerate}[{\rm(a)}]
\item $\zeta_N(s)$ has an analytic continuation
with isolated singularities to all $s\in\C$.
\item $\zeta_{N^{\ast}}(-s)=\zeta_{N}(s)(-1)^{N(1)}$
and $\eps_N(s)=(-1)^{N(1)}$.
\end{enumerate}
\end{proposition}

\begin{proof}
(a)\ We calculate
\begin{align*}
Z_N(w,s)
&=\sum_{\lambda,m}c(\lambda,m)\frac1{\Gamma(w)}
\int_1^\infty u^{-(s-\lambda)}(\log u)^{w+m-1}\frac{du}u\\
&=\sum_{\lambda,m}c(\lambda,m)\frac{\Gamma(w+m)}{\Gamma(w)}(s-\lambda)^{-w-m}.
\end{align*}
Hence
$$
\zeta_N(s)=\prod_{\lambda,m}\varphi_m(s-\lambda)^{c(\lambda,m)}
$$
with
\begin{align*}
\varphi_m(s)
&=\exp\(\left.\frac\partial{\partial w}\frac{\Gamma(w+m)}{\Gamma(w)}s^{-w-m}
\right|_{w=0}\)\\
&=\begin{cases}
\frac1s & (m=0)\\
\exp((m-1)!s^{-m}) & (m\ge1).
\end{cases}
\end{align*}
This gives the analytic continuation of $\zeta_N(s)$ to all $s\in\C$.

(b)\ Since
$$
N^\ast(u)=\sum_{\lambda,m}c(\lambda,m)(-1)^mu^{-\lambda}(\log u)^m,
$$
it holds that
$$
\zeta_{N^\ast}(s)=\prod_{\lambda,m}\varphi_m(s+\lambda)^{(-1)^mc(\lambda,m)}.
$$
Hence
$$
\eps_N(s)=\prod_{\lambda,m}\(\frac{\varphi_m(-s+\lambda)^{(-1)^m}}
{\varphi_m(s-\lambda)}\)^{c(\lambda,m)},
$$
where
$$
\frac{\varphi_m(-s)^{(-1)^m}}{\varphi_m(s)}
=\begin{cases}
-1 & (m=0),\\
1 & (m\ge0).
\end{cases}
$$
Thus
\begin{align*}
\eps_N(s)=\prod_\lambda(-1)^{c(\lambda,0)}=(-1)^{\sum_\lambda c(\lambda,0)}
=(-1)^{N(1)}.\tag*\qedhere
\end{align*}
\end{proof}

\begin{example}
Let $N(u)=\log u.$
Then $N(1)=0$ (admissible) and
\begin{align*}
\zeta_N(s)&=\exp\(\frac1s\),\\
\zeta_{N^\ast}(s)&=\exp\(-\frac1s\),
\end{align*}
and
$$\eps_N(s)=1.$$
\end{example}

\section{Determinants of Laplacians}
Another example of admissible counting functions is given by Laplace operators of compact manifolds as follows.
Let $M$ denote a compact Riemannian manifold with Laplacian $\Delta$.
Let $0=\la_0<\la_1\le\la_2\dots$ be the eigenvalues of $\Delta$ repeated according to their multiplicity.

\begin{proposition}
With the notation above,  the function 
$$
N(u)=\sum_{j=1}^\infty u^{-\la_j}
$$
is admissible and for $\Re(s)>0$ we have
$$
\zeta_N(-s)^{-1}={\det}'(\Delta+s),
$$
where $\det'$ denotes the regularized determinant in the sense of \cite{DHoker}.
\end{proposition}

\begin{proof}
We compute
\begin{align*}
Z_N(w,s)&=\frac1{\Gamma(w)}\int_1^\infty\frac{N(u)}{u^{s+1}}(\log u)^{w-1}du\\
&=\frac1{\Ga(w)}\int_0^\infty N(e^t)e^{-st}t^w\frac{dt}t\\
&=\sum_{j=1}^\infty\frac1{\Ga(w)}\int_0^\infty e^{-t(\la_j+s)}t^w\frac{dt}t\\
&=\sum_{j=1}^\infty (\la_j+s)^{-w}=\zeta_{\Delta+s}(w).
\end{align*}
The latter is the zeta function of the operator and the claim follows.
\end{proof}


\begin{bibdiv} \begin{biblist}

\bib{CCabsolutePoint}{article}{
   author={Connes, Alain},
   author={Consani, Caterina},
   title={Characteristic 1, entropy and the absolute point},
   conference={
      title={Noncommutative geometry, arithmetic, and related topics},
   },
   book={
      publisher={Johns Hopkins Univ. Press},
      place={Baltimore, MD},
   },
   date={2011},
   pages={75--139},
   review={\MR{2907005}},
}

\bib{DF1}{article}{
   author={Deitmar, Anton},
   title={Schemes over $\mathbb F_1$},
   conference={
      title={Number fields and function fields---two parallel worlds},
   },
   book={
      series={Progr. Math.},
      volume={239},
      publisher={Birkh\"auser Boston},
      place={Boston, MA},
   },
   date={2005},
   pages={87--100},
}

\bib{DzetaK}{article}{
   author={Deitmar, Anton},
   title={Remarks on zeta functions and $K$-theory over ${\bf F}_1$},
   journal={Proc. Japan Acad. Ser. A Math. Sci.},
   volume={82},
   date={2006},
   number={8},
   pages={141--146},
   issn={0386-2194},
}

\bib{toric}{article}{
   author={Deitmar, Anton},
   title={$\Bbb F_1$-schemes and toric varieties},
   journal={Beitr\"age Algebra Geom.},
   volume={49},
   date={2008},
   number={2},
   pages={517--525},
   issn={0138-4821},
   review={\MR{2468072 (2009j:14003)}},
}

\bib{DKK}{article}{
   author={Deitmar, Anton},
   author={Koyama, Shin-ya},
   author={Kurokawa, Nobushige},
   title={Absolute zeta functions},
   journal={Proc. Japan Acad. Ser. A Math. Sci.},
   volume={84},
   date={2008},
   number={8},
   pages={138--142},
   issn={0386-2194},
}

\bib{Deligne}{article}{
   author={Deligne, Pierre},
   title={La conjecture de Weil. I},
   language={French},
   journal={Inst. Hautes \'Etudes Sci. Publ. Math.},
   number={43},
   date={1974},
   pages={273--307},
   issn={0073-8301},
}
\bib{DHoker}{article}{
   author={D'Hoker, Eric},
   author={Phong, D. H.},
   title={On determinants of Laplacians on Riemann surfaces},
   journal={Comm. Math. Phys.},
   volume={104},
   date={1986},
   number={4},
   pages={537--545},
   issn={0010-3616},
}

\bib{Lor}{article}{
   author={Lorscheid, Oliver},
   title={Functional equations for zeta functions of $\Bbb F_1$-schemes},
   language={English, with English and French summaries},
   journal={C. R. Math. Acad. Sci. Paris},
   volume={348},
   date={2010},
   number={21-22},
   pages={1143--1146},
   issn={1631-073X},
   review={\MR{2738915 (2011j:14050)}},
   doi={10.1016/j.crma.2010.10.010},
}

\bib{RS}{article}{
   author={Ray, D. B.},
   author={Singer, I. M.},
   title={$R$-torsion and the Laplacian on Riemannian manifolds},
   journal={Advances in Math.},
   volume={7},
   date={1971},
   pages={145--210},
   issn={0001-8708},
   review={\MR{0295381 (45 \#4447)}},
}

\bib{Sarnak}{article}{
   author={Sarnak, Peter},
   title={Determinants of Laplacians},
   journal={Comm. Math. Phys.},
   volume={110},
   date={1987},
   number={1},
   pages={113--120},
   issn={0010-3616},
   review={\MR{885573 (89e:58116)}},
}

\bib{Soule}{article}{
   author={Soul{\'e}, Christophe},
   title={Les vari\'et\'es sur le corps \`a un \'el\'ement},
   journal={Mosc. Math. J.},
   volume={4},
   date={2004},
   number={1},
   pages={217--244, 312},
   issn={1609-3321},
}

\end{biblist} \end{bibdiv}
\end{document}